\documentclass[11pt]{amsart}

\usepackage[T1]{fontenc}
\usepackage[letterpaper,top=1.10in,bottom=0.5 in,left=1.25in,right=1.25in,footskip=.40in]{geometry}
\usepackage{amssymb,mathrsfs}
\usepackage{graphicx}
\usepackage{tikz}
\usetikzlibrary{arrows.meta,calc,positioning,fit,decorations.pathreplacing,backgrounds}
\usepackage{booktabs,array,ytableau}
\usepackage{enumitem}
\usepackage[font=small,labelfont=sc,labelsep=period]{caption}
\usepackage[hidelinks]{hyperref}
\usepackage[nameinlink,noabbrev]{cleveref}

\makeatletter
\let\ps@headings\ps@plain
\let\ps@firstpage\ps@plain
\makeatother

\numberwithin{equation}{section}
\allowdisplaybreaks[2]

\AtBeginDocument{%
	\setlength{\abovedisplayskip}{10pt plus 3pt minus 2pt}%
	\setlength{\belowdisplayskip}{10pt plus 3pt minus 2pt}%
	\setlength{\abovedisplayshortskip}{6pt plus 2pt minus 1pt}%
	\setlength{\belowdisplayshortskip}{6pt plus 2pt minus 1pt}%
}
\newtheorem{theorem}{Theorem}[section]
\newtheorem{lemma}[theorem]{Lemma}
\newtheorem{conj}[theorem]{Conjecture}
\newtheorem{prop}[theorem]{Proposition}
\newtheorem{corollary}[theorem]{Corollary}
\theoremstyle{remark}

\title{A combinatorial proof for the  positivity of the normalized  Jacobi triple product tails}
\author{Xiangyu Ding}
\address{Center for Combinatorics, LPMC, Nankai University, Tianjin 300071, P.~R. China}
\email{dingmath@mail.nankai.edu.cn}
\author{Lisa Hui Sun}
\address{Center for Combinatorics, LPMC, Nankai University, Tianjin 300071, P.~R. China}
\email{sunhui@nankai.edu.cn}
\date{}

\begin{document}
	
	\begin{abstract}
		For \(k\geq 1\), we prove that
		\[
		[q^n z^s]J_k(z,q)\geq 0,
		\qquad (n\geq 0,\ s\in\mathbb Z) 
		\]
		for the normalized Jacobi triple product tails
		\[
		J_k(z,q)=
		\frac{ \sum_{j=k}^{\infty}(-1)^{j-k}
			q^{\binom{j+1}{2}}(z^{-j}+\cdots+z^j)}
		{(zq,q/z;q)_\infty}.
		\] 
		This result not only implies Merca's stronger nonnegativity conjecture  on truncated Jacobi triple product series in full generality, but also yields infinite families of linear inequalities for two-colored partitions and partitions with parts in the residue classes 
		$\pm S \pmod{R}$. We present a combinatorial proof wherein a sign-reversing involution reduces the normalized Jacobi triple product tails to the  invariant subsets according to the generalized minimal--excludant of partitions. Furthermore, by combing an invertible lift operator on Frobenius arms with Konan's size- and length-preserving bijection, an injection is constructed between the consecutive invariant subsets, which implies the coefficientwise positivity of the normalized Jacobi triple product tails. 
	\end{abstract}
	
	\maketitle
	\thispagestyle{plain}
	
	\section{Introduction and main results}
	
	Euler's pentagonal-number theorem 
	\[
	\prod_{n=1}^{\infty}(1- q^n)=\sum_{j\in\mathbb Z}(-1)^j q^{j(3j-1)/2},
	\]
	stands as the foundational cornerstone of $q$-series and combinatorics. Multiplying it by $(q;q)_\infty^{-1}=\sum_{n\ge0}p(n)q^n$ yields the well-known  alternating recurrence relation for the partition function $p(n)$.
	Shanks \cite{Shanks1951} obtained a finite form of the pentagonal series in 1951. Researchers such as Warnaar\cite{Warnaar2002} and Chapman\cite{Chapman2003} have extensively explored other related finite forms and partial-sum identities.  The systematic study of coefficientwise positivity for such truncations began with the following  identity due to Andrews and Merca
	\cite{AndrewsMerca2012} 
	\[
	\frac{1}{(q;q)_\infty}
	\sum_{j=0}^{k-1}(-1)^j q^{j(3j+1)/2}(1-q^{2j+1})
	=
	1+(-1)^{k-1}\sum_{n\ge1}
	\frac{q^{\binom{k}{2}+(k+1)n}}{(q;q)_n}
	\genfrac{[}{]}{0pt}{}{n-1}{k-1}_{\!q}.
	\]
	They   interpreted the nonconstant coefficients as follows,
	\begin{equation}\label{eq:Mk-AM}
		(-1)^{k-1}\sum_{j=0}^{k-1}(-1)^j
		\bigl(p(n-j(3j+1)/2)-p(n-j(3j+5)/2-1)\bigr)=M_k(n),
	\end{equation}
	where  $M_k(n)$ is
	the number of partitions of $n$ in which $k$ is the least positive integer that
	does not occur as a part and the number of parts larger than $k$ is greater than
	the number of parts smaller than $k$.   
	This immediately yields the corresponding nonnegativity of the truncated
	pentagonal sums. We refer to \cite{KolitschBurnette2015,XiaZhao2022,Zhou2024} for the  subsequent bijective and analytic work 	clarified this interpretation which produced companion inequalities.
	
	Guo and Zeng \cite{GuoZeng2013} then established truncations of Gauss's identities on triangular numbers and 
	square numbers, which lead to inequalities for overpartitions and
	partitions with odd parts distinct.  Andrews and Merca \cite{AndrewsMerca2018} showed
	that both truncations given by Euler and Gauss  admit a unified treatment through the 	Rogers--Fine identity. Xia, Yee, and Zhao \cite{XiaYeeZhao2022} later
	constructed broader families of truncated forms for these three classical theta
	identities.  These developments made it natural to seek
	a truncation theory at the level of the Jacobi triple product, which reduces to  
	Euler and Gauss' identities as special cases.
	
	Andrews--Merca and Guo--Zeng independently conjectured that appropriately
	normalized finite truncations of this bilateral series have nonnegative
	coefficients.  More precisely, for $1\le S<R/2$ they asked whether every positive-degree
	coefficient of
	\begin{equation}\label{mercaconj}
		(-1)^{k-1} \frac{\sum_{j=0}^{k-1}(-1)^j q^{R j(j+1) / 2-S j}\left(1-q^{(2 j+1) S}\right)}{\left(q^S, q^{R-S}, q^R ; q^R\right)_{\infty}},
	\end{equation}
	is	nonnegative. Note that the complete series over $j\geq 0$ on the numerator  is just the Jacobi triple product seris which equals to 
	$(q^S,q^{R-S},q^R;q^R)_\infty$.  It have been certified analytically by   Mao \cite{Mao2015} and combinatorially by Yee in 2015
	\cite{Yee2015}.  Wang and Yee \cite{WangYee2019} subsequently obtained an explicit multiple-sum form based on  the
	Bailey pair methods and a $q$-series expansion of Liu \cite{Liu}.  The bilateral variants, averaged truncations, Hecke--Rogers analogues,
	and various combinatorial interpretations have since been developed in 	\cite{HeJiZang2016,JinLiuXia2025,SchlosserZhou2024,WangYee2020,Yao2025}.
	
	Merca observed that the preceding finite-truncation statement \eqref{mercaconj} can be
	strengthened substantially.  In  \cite{Merca2021,Merca2022}, Merca proposed a positivity conjecture
	after removing the factor	$(q^R;q^R)_\infty$ from the denominator of \eqref{mercaconj} and  extending the
	parameter range to  $1\le S<R$ .
	
	\begin{conj} \label{conj:merca-series}
		For positive integers $k, R, S$ with  $1 \leq S<R  $, the coefficient of $q^n$ with $n \geq 1$ in
		\begin{align}\label{conj-stronger}
			\mathcal T_{R,S,k}(q)=\frac{\sum_{j=k}^{\infty}(-1)^{j-k} q^{R j(j+1) / 2-S j}\left(1-q^{(2 j+1) S}\right)}{\left( q^S, q^{R-S}  ; q^R\right)_{\infty}}
		\end{align}
		is nonnegative.
	\end{conj}
	
	Ballantine and
	Feigon proved the first three truncation levels $k=1,2,3$ and obtained several
	related identities \cite{BallantineFeigon2025}.  Ding and Sun \cite{DingSunBailey2025} proved the
	specialization $R=3S$ by applying Bailey lattice.  Mao \cite{MaoAsymptotic2024} used Wright's circle method to obtain
	asymptotic formulas for the relevant Jacobi and quintuple-product tails,
	thereby proving the corresponding conjectures for sufficiently large
	coefficients.  In a separate analytic treatment,
	Ding and Sun \cite{DingSunPreprint2025} further established eventual positivity with an effective procedure for
	a threshold $N(R,S,k)$ and also proved the positivity when the truncation
	parameter is sufficiently large.
	Zhou  \cite{Zhou2024} obtained further tail-positivity results and
	specializations in a systematic study of pentagonal tails. Thereby this conjecture remains open in full generality  until now.

	The terminology ``tail" also carries a natural physical interpretation. In two-dimensional conformal field theory and solvable lattice models, $q$-series represent graded partition functions where $q$ tracks energy and $z$ acts as a charge fugacity \cite{AlvarezGaumeMooreVafa,DiFrancescoMathieuSenechal}. Truncating these series corresponds to a finite-volume restriction or an ultraviolet cutoff \cite{AndrewsBaxterForrester,MelzerCTM}. Consequently, the tail'' is the high-energy remainder of the spectrum. In this paper, we consider the normalized Jacobi triple product tails as follows
	\begin{equation}\label{eq:J-def}
		J_k(z,q)=
		\frac{ \sum_{j=k}^{\infty}(-1)^{j-k}q^{T_j}\chi_j(z)}
		{(zq,q/z;q)_\infty},
	\end{equation}
	where  $k\geq 1$,  $T_j=\binom{j+1}{2}$ and
	$\chi_j(z)=\sum_{u=-j}^{j}z^u$.   The Laurent variable $z$ records the difference between the numbers of  
	$A$-colored parts and $B$-colored parts in the  two colored-partitions.  Combinatorially, we prove the   coefficientwise positivity of $J_k(z,q)$.
	
	\begin{theorem}\label{thm:main}
		For every integer $k\ge1$,
		\begin{equation}\label{eqmain}
			J_k(z,q)\in\mathbb Z_{\ge0}[z,z^{-1}][[q]].
		\end{equation}
		Equivalently, $[q^{n}z^{s}]\,J_k(z,q)\ge0$ for all $n\ge0$ and
		$s\in\mathbb Z$.
	\end{theorem}
	
	The above theorem leads to inequalities on partition pairs and two-colored partitions immediately.Moreover, by substituting $(z,q)=(q^S,q^R)$ into the above theorem, it directly reduces to Merca's stronger conjecture  \ref{conj:merca-series}.  
	
	\begin{corollary}\label{cor:merca} Merca's stronger conjecture \ref{conj:merca-series}  holds.
	\end{corollary}
	
	Our proof is based on the minimal excludant of integer partitions, which is denoted by mex in short.  The statistic {\rm mex}  has long
	been used in combinatorial theory, see, for example, Fraenkel and
	Peled \cite{FraenkelPeled2015}.  Andrews and Newman \cite{AndrewsNewman2019,AndrewsNewman2020} introduced it
	systematically into partition theory and related its moments to classical
	partition functions.  The parity
	of the mex was then connected with Dyson's crank, and
	Hopkins, Sellers and Stanton \cite{HopkinsSellersStanton2022} formulated a shifted version that records the first
	missing part above a prescribed level, see also \cite{AndrewsNewman2019,AndrewsNewman2020,BallantineMerca2021,		HopkinsSellersStanton2022,HopkinsSellersYee2022,Konan2023}.
	
	For a partition $\lambda$,
	let
	\[
	\operatorname{mex}_r(\lambda)=\min\{t>r:t\text{ is not a part of }\lambda\},
	\]
	and let $\operatorname{Top}(\lambda)$ denote the set of the elements contained in the top row of its Frobenius symbol \cite{AndrewsPartitions}.  For $r\geq 0$, define
	\[
	\mathcal M_r=\{\lambda:\operatorname{mex}_r(\lambda)-r\text{ is odd}\},\qquad
	\mathcal F_r=\{\lambda:r\notin\operatorname{Top}(\lambda)\}.
	\]
	Konan \cite{Konan2023} established an explicit  bijection for these generalized mex classes. 
	
	\begin{theorem}[Konan \cite{Konan2023}]\label{thm:Konan}
		For every $r\ge0$ there is a bijection
		\[
		\mathsf K_r:\mathcal M_r\longrightarrow\mathcal F_r
		\]
		that preserves both size and length 
		\[
		|\mathsf K_r(\lambda)|=|\lambda|,\qquad
		\ell(\mathsf K_r(\lambda))=\ell(\lambda).
		\]
	\end{theorem}
	
	To give our combinatorial proof, first we construct a bidegree-preserving,
	sign-reversing involution on  set of quadruples.
	Then  the unpaired elements are organized into
	invariant subsets according to the mex   \(\mathcal C_r\) with the corresponding generating function $C_r(z,q)$. Based on Konan's bijection, we derive an 
	injection \(\Theta_r\) on this subset which  makes each negative block  in the invariant subset    is injected into a preceding positive block  by	\(\Theta_r\).   The argument therefore proves coefficientwise positivity of $J_k(z,q)$.

	This  paper is organized as follows.  Section \ref{chap:pre} introduces the basic definitions and 
	notations  for partitions and the formal power series.  Section \ref{chap:map} derives the
	invariant subset according to the mex decomposition combinatorially and records its direct
	\(q\)-series counterpart.  Section \ref{chap:first} constructs the
	Frobenius-arm lift and the injection between consecutive invariant subsets, thereby
	proving Theorem \ref{thm:main}.  Section \ref{chap:conclu} contains
	concluding remarks.  In Appendix \ref{chap:appendix}, we recall the detailed procedures of  Konan's bijection.

	\section{Partitions, minimal excludants, and colored indices}\label{chap:pre}
	
	In this paper, we assume that $|q|<1$ and all the power formal series are interpreted in 
	$\mathscr R=\mathbb Z[z,z^{-1}][[q]]$, where the coefficient of each $q^n$ is a
	Laurent polynomial in $z$.	Thereby,  for a
	fixed bidegree $(n,s)$ only finitely many summands can contribute to the coefficient of $q^nt^s$.
	
	Let
	\[
	(a;q)_n=\prod_{u=0}^{n-1}(1-aq^u),\qquad
	(a;q)_\infty=\prod_{n=0}^{\infty}(1-aq^n),
	\]
	and   $(a_1,a_2,\ldots,a_t;q)_\infty=\prod_{i=1}^t(a_i;q)_\infty$.
	
	We write $F\succeq G$, or equivalently $G\preceq F$, if every coefficient
	of $F-G$ is nonnegative. 
	
	A partition  \cite{AndrewsPartitions} of a nonnegative integer $n$ is a finite weakly decreasing sequence $\lambda_1\geq \lambda_2\geq \cdots \geq \lambda_k$ of positive integers such that $\sum_{i=1}^k \lambda_i=n$, which can be denoted by $\lambda=(\lambda_1, \lambda_2, \ldots, \lambda_k)$.
	The set of all partitions, including the empty partition, is denoted by
	$\mathcal P$.  For $\lambda \in\mathcal P$, let  $|\lambda|$, $\ell(\lambda)$, and
	$m_t(\lambda)$ denote its size, length, and multiplicity of the part $t$.
	Multiset union is written as $\lambda\sqcup\mu$.  Denote the staircase in the set of partitions by  
	\[
	\delta_r=(r,r-1,\ldots,1),\qquad \delta_0=\varnothing,
	\]
	so that $|\delta_r|=T_r$, $\ell(\delta_r)=r$ and we call $r$ the level of this staircase. 
	
	If $d$ is the size of the Durfee square of $\lambda=(\lambda_1, \lambda_2, \ldots, \lambda_k)$, that is, $d=\max\{i\geq 1 \colon \lambda_i\geq i\}$,   and $\lambda'$ is its
	conjugate, then the Frobenius symbol of $\lambda$ is given by \cite{AndrewsPartitions}
	\[
	\lambda\longleftrightarrow
	\begin{pmatrix}
		a_1&a_2&\cdots&a_d\\
		b_1&b_2&\cdots&b_d
	\end{pmatrix},
	\]
	where $ a_i=\lambda_i-i$ and $b_i=\lambda'_i-i$ for $1\leq i \leq d$. The two rows are strictly decreasing sequences of nonnegative integers, which represent  the Frobenius arms and legs, respectively, such that 
	\begin{equation}\label{Frobeniusc}
		|\lambda|=d+\sum_{i=1}^d(a_i+b_i),
		\qquad	\ell(\lambda)=b_1+1\quad(\lambda\ne\varnothing).
	\end{equation}
	Denote  $\operatorname{Top}(\lambda)=\{a_1,\ldots,a_d\}$ and
	$\operatorname{Top}(\varnothing)=\varnothing$.	See Figure \ref{fig:frobenius} for an example.
	
	
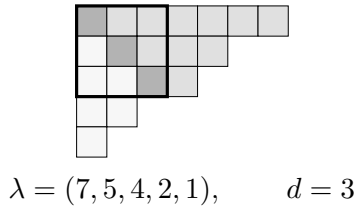
\begin{figure}[h]
		\centering
		\begin{minipage}[c]{.34\textwidth}
			\centering
			\begin{tikzpicture}[x=.40cm,y=-.40cm]
				\foreach \rr/\len in {1/7,2/5,3/4,4/2,5/1}{
					\foreach \cc in {1,...,\len}{
						\ifnum\rr=\cc
						\filldraw[fill=black!30,draw=black,line width=.3pt]
						(\cc-1,\rr-1) rectangle (\cc,\rr);
						\else
						\ifnum\cc>\rr
						\filldraw[fill=black!12,draw=black,line width=.3pt]
						(\cc-1,\rr-1) rectangle (\cc,\rr);
						\else
						\filldraw[fill=black!3,draw=black,line width=.3pt]
						(\cc-1,\rr-1) rectangle (\cc,\rr);
						\fi
						\fi
					}
				}
				\draw[very thick] (0,0) rectangle (3,3);
			\end{tikzpicture}
			
			\smallskip
			$\lambda=(7,5,4,2,1)$, \qquad $d=3$
		\end{minipage}
		\hfill
		\begin{minipage}[c]{.57\textwidth}
			\centering
			\[
			a_i=\lambda_i-i,\qquad b_i=\lambda_i'-i,
			\]
			\[
			\lambda\longleftrightarrow
			\begin{pmatrix}6&3&1\\4&2&0\end{pmatrix}.
			\]
			\small
			The cells  enclosed by the black border form the Durfee square.  The lightly shaded cells  give the Frobenius arms 6, 3, 1; the cells
			below the diagonal  give the legs 4, 2, 0.
		\end{minipage}
		\caption{The Young diagram of a partition and its Frobenius symbol.}
		\label{fig:frobenius}
	\end{figure}

	For $\lambda\in\mathcal P$, define the ordinary and shifted minimal excludants by
	\[
	\operatorname{mex}(\lambda)=\min\{t\ge1:m_t(\lambda)=0\},\qquad
	\operatorname{mex}_{r}(\lambda)=\min\{t>r:m_t(\lambda)=0\}.
	\]
	For example, if $\lambda=(7,5,4,2,1)$, then $\operatorname{mex}_1(\lambda)=\operatorname{mex}_2(\lambda)=3$.
	
	Recall that 
	\[
	\mathcal M_r=\{\lambda\in\mathcal P:\operatorname{mex}_{r}(\lambda)-r\text{ is odd}\},
	\qquad
	\mathcal F_r=\{\lambda\in\mathcal P:r\notin\operatorname{Top}(\lambda)\}.
	\]
	The condition defining $\mathcal F_r$ concerns a Frobenius arm, not the presence
	of $r$ as an ordinary part. For example,  $\lambda=(7,5,4,2,1)$, then
	$\operatorname{Top}(7,5,4,2,1)=(6,3,1)$, hence it lies in $\mathcal F_r$ precisely when
	$r\notin\{1,3,6\}$.

	Assume that for a partition with two colors $A$ and $B$, an \( A \)-colored part of size \( i \), denoted by \( i^A \), has weight \( q^i z \), whereas a \( B \)-colored part of size \( i \), denoted by \( i^B \), has weight \( q^i z^{-1} \). Note that 
	the superscripts $A$ and $B$ are just color labels, not powers. Thus a pair \( (\alpha, \beta) \in \mathcal{P}^2 \), where \( \alpha \) records the \( A \)-colored parts and \( \beta \)  records the \( B \)-colored parts, has weight
	\[
	\operatorname{wt}(\alpha, \beta) = q^{|\alpha| + |\beta|} z^{\ell(\alpha) - \ell(\beta)}.
	\]
	Consequently,
	\begin{equation}\label{wtab}
		\sum_{\alpha, \beta \in \mathcal{P}} \operatorname{wt}(\alpha, \beta) = \frac{1}{(zq, q/z; q)_\infty}.
	\end{equation}
	We also set the weight of an uncolored atom $U$  to be $zq^0=z$.
	For example,  for a two colored partition pair $(\alpha=(5,2,2), \beta=(4,3,1))$, we have $\operatorname{wt}(\alpha,\beta)=q^{17}z^{3-3}=q^{17}$.

	\section{Preliminary cancellations and the invariant subset decomposition}\label{chap:map}
	
	We first interpret  $J_k(z,q)$, both combinatorially and analytically. 	
	Since
	\[
	\chi_j(z)=z^{-j}(1+z+\cdots+z^{2j}),
	\]
	equation \eqref{eq:J-def} can be rewritten as
	\[
	J_k(z,q)=\sum_{j=k}^{\infty}(-1)^{j-k}
	\frac{q^{T_j}z^{-j}(1+z+\cdots+z^{2j})}
	{(zq,q/z;q)_\infty},
	\]
	where the factor \(q^{T_j}z^{-j}\) can be seen as the weight of a \(B\)-colored staircase
	\[
	\delta_j=(j,j-1,\ldots,1).
	\]
	Let \(\mathcal X_k\) be the set of quadruples
	\[
	\mathcal{X}_k=\{y=(j,m,\alpha,\beta)\colon  \alpha, \beta \in \mathcal{P}, \ j\ge k,\  0\le m\le 2j,\   {\rm mex}(\beta)>j\},
	\]
	where $\alpha, \beta$ are colored by $A$ and $B$, respectively. Define
	\[
	\operatorname{sgn}(y)=(-1)^{j-k},
	\qquad
	\operatorname{wt}(y)
	=q^{|\alpha|+|\beta| }
	z^{m+\ell(\alpha)-\ell(\beta)  }.
	\]
	Here \(m\) records the choice of the monomial \(z^m\) in
	\(1+z+\cdots+z^{2j}\), or equivalently \(m\) copies of the
	atom \(U\).  Combining with \eqref{wtab}, it is easy to see that 
	\[
	J_k(z,q)
	=
	\sum_{y\in\mathcal X_k}
	\operatorname{sgn}(y)\operatorname{wt}(y).
	\]
	
	For \(\beta\in\mathcal P\), let
	\[
	h(\beta)
	:=
	\max\{t\ge0:\delta_t\subseteq\beta\}
	=
	\operatorname{mex}(\beta)-1,
	\]
	where the staircase containment is understood with multiplicity. For \(m\ge0\), set
	\[
	L(m)
	:=
	\max\left\{k,\left\lceil\frac m2\right\rceil\right\}.
	\]
	Once \(m,\alpha,\beta\) are given, the admissible values of \(j\)
	are precisely the integers in the interval
	\[
	L(m)\le j\le h(\beta).
	\]
	
	We now define a map
	\begin{align*}
		\iota:&\mathcal X_k\longrightarrow\mathcal X_k      
	\end{align*}
	by 
	\[
	\iota(j,m,\alpha,\beta)
	=
	\begin{cases}
		(j+1,m,\alpha,\beta),
		& j-L\ \text{is even and }j<h,\\[2mm]
		(j-1,m,\alpha,\beta),
		& j-L\ \text{is odd},\\[2mm]
		(j,m,\alpha,\beta),
		& j=h\ \text{and }h-L\ \text{is even},
	\end{cases}
	\]
	where \(y=(j,m,\alpha,\beta)\in \mathcal{\chi}_k\), and 
	$L=L(m), h=h(\beta)$.
	
	\begin{prop}
		The map \(\iota\) is an involution, and the invariant subset of this involution is 
		\[
		\mathcal I_{\iota}
		=
		\left\{
		(j,m,\alpha,\beta)\in\mathcal{\chi}_k:
		j=h(\beta),\ \operatorname{mex}(\beta)=L(m)+2d+1
		\text{ for some }d\ge0
		\right\}.
		\]
		
	\end{prop}
	
	\begin{proof}
		In the first case,
		\(j<h\) implies
		\[
		\delta_{j+1}\subseteq\beta.
		\]
		In the second case, \(j-L\) is odd, so \(j\ge L+1\). Hence
		\[
		j-1\ge L\ge k,\qquad \delta_{j-1}\subseteq \beta \qquad
		\text{and}\qquad
		m\le 2L\le 2(j-1).
		\]	
		
		For any element in the first case, 
		the staircase under $\iota$ is mapped to of level \(j+1\) which has odd difference from \(L\), so applying $\iota$  again simply maps it back to itself at level $j$. Conversely, for any element in the second
		case, applying $\iota$ twice clearly maps it back to itself at level $j$. The elements in the third case consist  precisely of the fixed points of $\iota$. Moreover, since \(\iota\) changes only \(j\), it preserves
		\(\operatorname{wt}(y)\). On every nonfixed orbit, \(j\) is changed by
		one, and therefore
		\[
		\operatorname{sgn}(\iota(y))
		=
		-\operatorname{sgn}(y).
		\]
		Thus \(\iota\) is a weight-preserving sign-reversing involution.
		
		Note that an element \(y=(j,m,\alpha,\beta)\) is fixed under \(\iota\) if and only if
		\[
		j=h(\beta),
		\qquad
		h(\beta)-L(m)\equiv 0 \pmod2,
		\]
		and $\operatorname{sgn}(y)	=(-1)^{L(m)-k}$.
		Equivalently, there exists \(d\ge0\) such that
		\[
		\operatorname{mex}(\beta)
		=
		L(m)+2d+1,
		\]
		which completes the proof. 
	\end{proof}
	
	For \(r\ge1\), define the invariant subset according to the mex
	\begin{equation}\label{invariantab}
		\mathcal C_r
		=
		\left\{
		(\alpha,\beta)\in\mathcal P^2:
		\operatorname{mex}(\beta)=r+2d+1
		\text{ for some }d\ge0
		\right\},
	\end{equation}
	and let
	\[
	C_r(z,q)
	=
	\sum_{(\alpha,\beta)\in\mathcal C_r}
	q^{|\alpha|+|\beta|}
	z^{\ell(\alpha)-\ell(\beta)}.
	\]
	By applying the involution $\iota$, we obtain the following decomposition of $J_k(z,q)$ according to the invariant subset of $\iota$. Here we present both combinatorial proof and anlaytic proof of such a decomposition. 
	
	\begin{lemma}[Invariant subset decomposition of $J_k(z,q)$] \label{lem:block}
		For every \(k\ge1\),
		\begin{align}
			J_k(z,q)
			={}&(1+z+\cdots+z^{2k})C_k(z,q)\notag\\
			&+\sum_{r=k+1}^{\infty}
			(-1)^{r-k}z^{2r-1}(1+z)C_r(z,q).
			\label{eq:com-decomposition}
		\end{align}
	\end{lemma}
	
	\begin{proof}[Combinatorial proof.]
		Under the involution \(\iota\), only its fixed points remain. Set \(r=L(m)\).
		
		If \(r=k\), then \(0\le m\le2k\), then every surviving element has
		positive sign, and the coresponding  \((\alpha,\beta)\) belongs to
		\(\mathcal C_k\). Summing over the possible values of \(m\) gives
		\[
		(1+z+\cdots+z^{2k})C_k(z,q).
		\]
		
		If \(r>k\), then \(L(m)=\lceil m/2\rceil=r\), and hence
		\[
		m\in\{2r-1,2r\}.
		\]
		The sign of the fixed point is \((-1)^{r-k}\), while the two possible
		values of \(m\) contribute to
		\[
		z^{2r-1}+z^{2r}
		=
		z^{2r-1}(1+z).
		\]
		Summing over all \(r\ge k+1\) gives
		\eqref{eq:com-decomposition}. Conversely, every \( (\alpha, \beta) \in \mathcal{C}_r \), together with an admissible value of \( m \) described above,
		determines the unique fixed point \( j = h(\beta) \).
	\end{proof}

	\begin{proof}[Analytic proof.]
		Formula  \eqref{eq:com-decomposition} also follows from \eqref{eq:J-def}.  Let
		\[
		F_j(z,q)=\frac{q^{T_j}z^{-j}}{(zq,q/z;q)_\infty},
		\]
		and let
		\[
		A_j(z)=1+z+\cdots+z^{2j},\qquad
		D_r(z)=z^{2r-1}(1+z).
		\]
		Then
		\[
		J_k(z,q)=\sum_{j=k}^{\infty}(-1)^{j-k}A_j(z)F_j(z,q),
		\qquad
		A_j(z)=A_k(z)+\sum_{r=k+1}^{j}D_r(z).
		\]
		Collecting equal blocks in the region $k+1\le r\le j$ gives
		\[
		\begin{aligned}
			J_k(z,q)
			= {}&A_k(z)\sum_{j=k}^{\infty}(-1)^{j-k}F_j(z,q)+\sum_{r=k+1}^{\infty}D_r(z)
			\sum_{j=r}^{\infty}(-1)^{j-k}F_j(z,q)\\
			={}&A_k(z)  C_k(z,q)
			+\sum_{r=k+1}^{\infty}(-1)^{r-k}D_r(z)\widehat C_r(z,q),
		\end{aligned}
		\]
		where 
		\[
		\widehat C_r(z,q)=\sum_{j=r}^{\infty}(-1)^{j-r}F_j(z,q).
		\]
		Pairing up the consecutive summands, we have 
		\[
		\widehat C_r(z,q)=\sum_{d\ge0}
		\bigl(F_{r+2d}(z,q)-F_{r+2d+1}(z,q)\bigr).
		\]
		Noting that $T_{j+1}-T_j=j+1$, it implies that 
		\[
		\begin{aligned}
			F_j(z,q)-F_{j+1}(z,q)
			&=\frac{q^{T_j}z^{-j}(1-q^{j+1}z^{-1})}{(zq,q/z;q)_\infty}\\
			&=\frac1{(zq;q)_\infty}
			\left(\prod_{i=1}^{j}\frac{q^iz^{-1}}{1-q^iz^{-1}}\right)
			\left(\prod_{i=j+2}^{\infty}\frac1{1-q^iz^{-1}}\right).
		\end{aligned}
		\]
		The first factor $1/(zq;q)_\infty$ permits the generating function of  arbitrary $A$-colored partitions.  According to the following two brackets, the $B$-colored partitions must have parts 
		$1, 2, \ldots,j$ occurring at least once, $j+1$ being absent, and all larger
		parts being unrestricted.  Thus $F_j(z,q)-F_{j+1}(z,q)$  is the generating function
		for pairs $(\alpha, \beta)$ with ${\rm{mex}}(\beta)=j+1$, that is to say, $j+1$  is the first missing $B$–colored part.  Taking $j=r+2d$ shows that 
		$\widehat C_r(z,q)=C_r(z,q)$, and this proves
		\eqref{eq:com-decomposition}. 
	\end{proof}

	\section{Frobenius operation and the combinatorial proof}\label{chap:first}
	
	In this section, we will construct a Frobenius-arm lift operator as well as an injection which can be combined with  Konan's  bijection to build an injection from the invariant subset $\mathcal C_r$ to $\mathcal C_{r-1}$.

	Taking $(\alpha,\beta)\in\mathcal C_r$,  since $\operatorname{mex}(\beta)>r$, by removing one copy of each
	part among $1,2,\ldots,r$, it yields  a unique decomposition for $\beta$:
	\[
	\beta=\delta_r\sqcup\gamma,
	\]
	such that 
	$\operatorname{mex}_{r}(\gamma)=\operatorname{mex}(\beta)=r+2d+1$, and hence $\gamma\in\mathcal M_r$.
	Conversely, adjoining $\delta_r$ to any $\gamma\in\mathcal M_r$ produces a
	$B$-colored partition in the invariant subset $\mathcal C_r$.
	
	To give the combinatorial proof of our main result, we need to introduce the a Frobenius operation associated with  the following sets 
	\[
	D_r:=\{\lambda\in\mathcal F_r:r-1\in \operatorname{Top}(\lambda)\},
	\qquad
	E_r:=\{\mu\in\mathcal F_{r-1}:r\in \operatorname{Top}(\mu)\}.
	\]
	\begin{lemma}\label{lem:Ur}
		For a given $r\ge2$, there exists a bijection
		\[
		\mathsf U_r:D_r\longrightarrow E_r,
		\]
		where for $\lambda\in D_r$,  $\mathsf U_r(\lambda)$ is obtained by replacing the unique top-row entry $r-1$ with $r$ and
		leaving the bottom row unchanged.
	\end{lemma}

	\begin{proof} For a partition $\lambda \in D_r$, suppose its Frobenius symbol is as follows 
		\[
		\lambda=
		\begin{pmatrix}
			a_1&\cdots&a_d\\
			b_1&\cdots&b_d
		\end{pmatrix},
		\]
		and let $a_t=r-1$. Since $\lambda\in\mathcal F_r$, the integer $r$
		does not occur in the top row. Hence
		\[
		a_{t-1}\ge r+1\quad(t>1),
		\qquad
		a_{t+1}\le r-2\quad(t<d).
		\]
		Thus replacing $a_t=r-1$ with $r$ preserves the strictly decreasing order of
		the top row, and thereby  
		$\mathsf U_r(\lambda)\in E_r$.
		
		Conversely, if $\mu\in E_r$, then its top row contains a unique entry
		$r$ and contains no entry $r-1$. Replacing that entry $r$ with $r-1$
		again preserves the strictly decreasing order and produces a partition in $D_r$.
		This lowering operation is the inverse of $\mathsf U_r$. Thus
		$\mathsf U_r $	is a bijection.
	\end{proof}
	
	We call $\mathsf U_r$ the Frobenius-arm lift operator since it  extends one of the arms of $\lambda$ by 1.  From the properties of the Frobenius symbol \eqref{Frobeniusc},
	it is clear that 
	\[
	|\mathsf U_r(\lambda)|=|\lambda|+1,
	\qquad
	\ell(\mathsf U_r(\lambda))=\ell(\lambda).
	\]
	Figure \ref{fig:arm-lift} illustrates this operation.
	\begin{figure}[h]
		\centering
		\resizebox{.74\textwidth}{!}{%
			\begin{tikzpicture}[x=.38cm,y=-.38cm,>=Latex]
				\begin{scope}
					\foreach \rr/\len in {1/7,2/5,3/4,4/2,5/1}{
						\foreach \cc in {1,...,\len}{
							\filldraw[fill=black!4,draw=black,line width=.28pt]
							(\cc-1,\rr-1) rectangle (\cc,\rr);
						}
					}
					\node at (3.5,-.72) {$(7,5,4,2,1)$};
					\node at (3.5,5.62) {$\operatorname{Top}=\{6,3,1\}$};
				\end{scope}
				\draw[->,very thick] (8.0,2) -- (16.8,2)
				node[midway,above,align=center] {$\mathsf U_4$:\\$3\mapsto4$};
				\begin{scope}[xshift=8cm]
					\foreach \rr/\len in {1/7,2/6,3/4,4/2,5/1}{
						\foreach \cc in {1,...,\len}{
							\ifnum\rr=2
							\ifnum\cc=6
							\filldraw[fill=black!30,draw=black,line width=.28pt]
							(\cc-1,\rr-1) rectangle (\cc,\rr);
							\else
							\filldraw[fill=black!4,draw=black,line width=.28pt]
							(\cc-1,\rr-1) rectangle (\cc,\rr);
							\fi
							\else
							\filldraw[fill=black!4,draw=black,line width=.28pt]
							(\cc-1,\rr-1) rectangle (\cc,\rr);
							\fi
						}
					}
					\node at (3.5,-.72) {$(7,6,4,2,1)$};
					\node at (3.5,5.62) {$\operatorname{Top}=\{6,4,1\}$};
				\end{scope}
		\end{tikzpicture}}
		\caption{The Frobenius-arm lift $\mathsf U_4$.  One cell is added to the arm
			corresponding to the top-row entry $3$; the bottom row and the partition
			length are unchanged.}
		\label{fig:arm-lift}
	\end{figure}
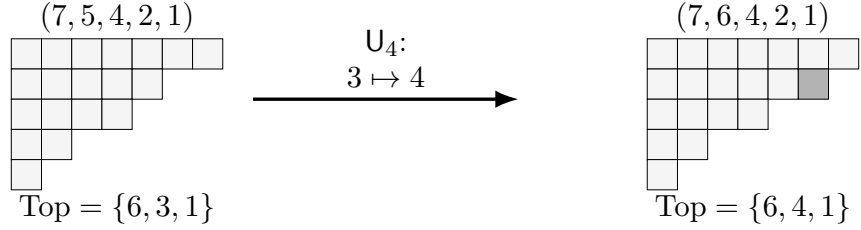
	
	Let $\widehat{r}$ denote a  marked $A$-colored part of size $r$. 
	The next lemma  establishes an injection.
	
	\begin{lemma}\label{lem:B-map}
		The map
		\[
		\mathsf B_r:\mathcal P\times\{\widehat{r}\}\times\mathcal F_r
		\hookrightarrow \mathcal P\times\mathcal F_{r-1}
		\]
		with 
		\[
		\mathsf B_r(\alpha,\widehat{r},\lambda)=
		\begin{cases}
			(\alpha\sqcup(r),\lambda),&\mbox{if} \ r-1\notin\operatorname{Top}(\lambda),\\[1mm]
			(\alpha\sqcup(r-1),\mathsf U_r(\lambda)),
			& \mbox{if}\ r-1\in\operatorname{Top}(\lambda) 
		\end{cases}
		\]
		is an injection for $r\geq 2$.  More precisely, it is a bijection onto the disjoint union
		$\mathcal I_r=\mathcal I_r^{(0)}\sqcup\mathcal I_r^{(1)}\subset\mathcal P\times\mathcal F_{r-1}$, where
		\[
		\begin{aligned}
			\mathcal I_r^{(0)}&=\{(\alpha',\mu):
			r\notin\operatorname{Top}(\mu),\ m_r(\alpha')\ge1\},\\
			\mathcal I_r^{(1)}&=\{(\alpha',\mu):
			r\in\operatorname{Top}(\mu),\ m_{r-1}(\alpha')\ge1\},
		\end{aligned}
		\]
		and in both branches $		
		|\alpha'|+|\mu|=|\alpha|+r+|\lambda|$, $	\ell(\alpha')=\ell(\alpha)+1$, and $		\ell(\mu)=\ell(\lambda)$. 
	\end{lemma}
	
	\begin{proof}
		It is obvious that $\mathcal I_r^{(0)}\cap \mathcal I_r^{(1)}=\emptyset$. The first branch of $\mathsf B_r(\alpha,\widehat{r},\lambda)$ lands in $\mathcal I_r^{(0)}$: the output top row still omits       
		$r$, and the newly inserted part $r$ ensures that $m_r(\alpha')\geq1$.  The second
		branch lands in $\mathcal I_r^{(1)}$: the arm lift operator puts $r$ into the output top row
		and $\alpha'$ contains the newly inserted part $r-1$.  
		
		Conversely, for
		$(\alpha',\mu)\in\mathcal I_r$, set
		\[
		(\alpha,\widehat{r},\lambda)=
		\begin{cases}
			(\alpha'\setminus(r),\widehat{r},\mu),
			&r\notin\operatorname{Top}(\mu),\\[1mm]
			(\alpha'\setminus(r-1),\widehat{r},\mathsf U_r^{-1}(\mu)),
			&r\in\operatorname{Top}(\mu).
		\end{cases}
		\]
		Here $\alpha'\setminus(t)$ denotes the partition obtained by removing one copy of the part $t$.
		In the first case, $(\alpha',\mu)\in\mathcal I_r^{(0)}$ thus $\mu\in\mathcal F_{r-1}$ and its top row also omits $r$. So we have $\lambda=\mu\in\mathcal F_r$ and thereby $(\alpha,\widehat{r},\lambda) \in \mathcal P\times\{\widehat{r}\}\times\mathcal F_r$.  In the second case, it is also true since  $\lambda=\mathsf U_r^{-1}(\mu)$ is again an element of $\mathcal F_r$.  Therefore, $\mathsf B_r$ is a bijection onto the set $\mathcal I_r$.
		
		We can see that the size and length statements are obvious from the definition of $\mathsf U_r$. Thus, the proof is complete.
	\end{proof}
	
	To show the map $\mathsf B_r$  is not onto   $\mathcal P\times\mathcal F_{r-1}$, we consider 
	$(\varnothing,\varnothing)$ which lies outside its image.

	Empolying the maps $\mathsf U_r$, $\mathsf B_r$,  and Konan's bijection, we can finish  the proof of our main result \eqref{eqmain}.
	
	\begin{theorem}[Invariant subset injection]\label{thm:core-injection}
		For every  $r\ge2$,  there is an injection
		\begin{align*}
			\Theta_r:& \{(U,U)\}\times\mathcal C_r\hookrightarrow\mathcal C_{r-1}\\
			& (U,U,\alpha,\beta) \mapsto (\alpha',\beta')
		\end{align*}
		where $(\alpha,\beta)\in \mathcal{C}_r$, $
		\beta=\delta_r\sqcup\gamma$, with 
		$ \gamma\in\mathcal M_r$,  and 
		$\lambda=\mathsf K_r(\gamma)\in\mathcal F_r$ is obtained  by using   Konan's bijection as given in Theorem \ref{thm:Konan}, then $\alpha'$ is determined by \[
		\mathsf B_r(\alpha,\widehat{r},\lambda)=(\alpha',\mu),
		\qquad \mu\in\mathcal F_{r-1},
		\]
		$\beta'$ is given by \[
		\beta'=\delta_{r-1}\sqcup\gamma'.
		\] with $\gamma'=\mathsf K_{r-1}^{-1}(\mu)\in\mathcal M_{r-1}$, such that 
		\[
		|\alpha'|+|\beta'|=|\alpha|+|\beta|,
		\qquad
		\ell(\alpha')-\ell(\beta')
		=\ell(\alpha)-\ell(\beta)+2.
		\]
		Consequently,
		\begin{equation}\label{eq:core-coeff-ineq}
			z^2C_r(z,q)\preceq C_{r-1}(z,q).
		\end{equation}
	\end{theorem}
	
	\begin{proof}
		By replacing the staircase	$\delta_r$ with $\delta_{r-1}$. the removed part $r$ is marked by color $B$ of   weight	$q^rz^{-1}$.  Together with the two atoms $(U,U)$, it has weight
		\[
		(q^rz^{-1})z^2=q^rz,
		\]
		which can be seen as the weight of a marked $A$-colored part $\widehat{r}$.   Since $\gamma'\in\mathcal M_{r-1}$, there exists $d^{\prime}\geq0$ such that  $\operatorname{mex}(\beta') = \operatorname{mex}_{r-1}(\gamma') = (r-1) + 2d^{\prime} + 1$. Combine it with the fact that $\alpha'$ is an arbitrary $A$-colored partition containing part $r$ or $r-1$ and check the definition of $\mathcal{C}_{r}$ in \eqref{invariantab}, we have $(\alpha',\beta')\in \mathcal C_{r-1}$.
		
		To prove the injectivity, we construct an explicit left  inverse of
		$\Theta_r$ on its image.  Let
		\[
		(\alpha',\beta')\in\operatorname{Im}(\Theta_r).
		\]
		Since $(\alpha',\beta')\in\mathcal C_{r-1}$, the partition $\beta'$
		contains at least one copy of each of the parts
		$1,2,\ldots,r-1$.  Hence the separation operation
		\[
		\gamma':=\beta'\setminus\delta_{r-1}
		\]
		is uniquely determined.  Moreover,
		\[
		\operatorname{mex}_{r-1}(\gamma')
		=\operatorname{mex}(\beta')
		=r+2d
		\]
		for some $d\ge0$, and therefore
		$\gamma'\in\mathcal M_{r-1}$.  We may thus recover uniquely
		\[
		\mu:=\mathsf K_{r-1}(\gamma')\in\mathcal F_{r-1}.
		\]
		
		Because $(\alpha',\beta')$ lies in the image of $\Theta_r$, 
		the pair	$(\alpha',\mu)$ lies in the image of the marked-part injection $\mathsf B_r$.
		Lemma \ref{lem:B-map} shows that the branch of $\mathsf B_r$ used in the  construction is determined	uniquely by the Frobenius top row of $\mu$.
		
		If
		$r\notin\operatorname{Top}(\mu),$
		then $(\alpha',\mu)\in\mathcal I_r^{(0)} $ which admits the first branch of $\mathsf B_r$.  In particular,
		$\alpha'$ contains a part $r$, and the unique predecessor pair is
		\[
		\alpha=\alpha'\setminus(r),
		\qquad
		\lambda=\mu.
		\]		 
		
		If $r\in\operatorname{Top}(\mu),$
		then $(\alpha',\mu)\in\mathcal I_r^{(1)} $  which admits the second branch of $\mathsf B_r$.  Since
		$\mu\in\mathcal F_{r-1}$, we know its Frobenius top row contains no entry
		$r-1$.
		Further note that $\alpha'$ contains a part $r-1$, and in this case the
		unique predecessor pair is
		\[
		\alpha=\alpha'\setminus(r-1),
		\qquad
		\lambda=\mathsf U_r^{-1}(\mu).
		\]
		Thus in either case the triple
		$(\alpha,\widehat r,\lambda)$ is recovered uniquely from
		$(\alpha',\mu)$ and $\lambda\in\mathcal F_r$.
		
		Finally, applying the inverse map of Konan's  bijection at level $r$, we obtain
		\[
		\gamma:=\mathsf K_r^{-1}(\lambda)\in\mathcal M_r,
		\qquad
		\beta:=\delta_r\sqcup\gamma,
		\]
		and therefore   a unique element
		$ 
		(U,U;\alpha,\beta)\in \{(U,U)\}\times\mathcal C_r$ is derived. 
		Denoting this reverse map by $\Psi_r$, together with $\Theta_r$ gives
		\[
		\Psi_r\!\left(\Theta_r(U,U;\alpha,\beta)\right)
		=(U,U;\alpha,\beta).
		\]
		Hence $\Psi_r$ is a left inverse of
		$\Theta_r$, and thereby $\Theta_r$ is injective.
		
		Konan's map  preserves size and length, while Lemma \ref{lem:B-map} gives
		\[
		|\alpha'|+|\gamma'|=|\alpha|+r+|\gamma|,
		\qquad
		\ell(\alpha')=\ell(\alpha)+1,
		\qquad
		\ell(\gamma')=\ell(\gamma).
		\]
		Then by the relations $|\delta_r|-|\delta_{r-1}|=r$ and
		$\ell(\delta_r)-\ell(\delta_{r-1})=1$, it yields the two conditions
		\[
		|\alpha'|+|\beta'|=|\alpha|+|\beta|,
		\qquad
		\ell(\alpha')-\ell(\beta')
		=\ell(\alpha)-\ell(\beta)+2.
		\]
		After the two atoms $(U,U)$ are included, the initial element  $ (U,U;\alpha,\beta)$  and its image $\Theta_r(U,U;\alpha,\beta)$ contribute to the terms of the same bidegree  in the generating function, and thus the injection gives \eqref{eq:core-coeff-ineq}.
	\end{proof}

	We now use an example to illustrate each step of the injection $\Theta_r$. 	Let
	\[
	r=2,\qquad \alpha=\varnothing,\qquad
	\beta=(4,4,2,1),
	\]
	and take the two  atoms \((U,U)\) into consideration.  Since $
	\operatorname{mex}(\beta)=3$,
	we have \((\alpha,\beta)\in\mathcal C_2\).
	\begin{itemize}
		\item[(1)] \emph{Remove the   staircase $\delta_r$ from $\beta$.}
		The staircase at level \(2\) is \(\delta_2=(2,1)\), and hence
		\[
		\beta=\delta_2\sqcup\gamma,
		\qquad
		\gamma=(4,4).
		\]
		Moreover, \(\operatorname{mex}_2(\gamma)=3\), so
		\(\gamma\in\mathcal M_2\).
		
		\item[(2)]	 \emph{Encode by Konan's bijection on $\gamma$.} Following the detailed procedures as given in Appendix \ref{chap:appendix}, it gives
		\[
		\lambda=\mathsf K_2(\gamma)=\mathsf K_2(4,4)=(5,3).
		\]
		By the Frobenius symbol, we have
		\[
		(4,4)\longleftrightarrow
		\begin{pmatrix}3&2\\1&0\end{pmatrix},\qquad
		(5,3)\longleftrightarrow
		\begin{pmatrix}4&1\\1&0\end{pmatrix}.
		\]
		Thus \(\operatorname{Top}(\lambda)=\{4,1\}\) avoids \(2\), as required for
		\(\lambda\in\mathcal F_2\).  The encoding of Konan's map preserves both size and length:
		\(|\lambda|=|\gamma|=8\) and \(\ell(\lambda)=\ell(\gamma)=2\).
		
		\item[(3)] \emph{Convert the removable $B$-colored part and the two atoms into a
			marked \(A\)-colored part $\widehat{r}$.}
		Replace \(\delta_2\) with \(\delta_1=(1)\) and removes the
		\(B\)-colored part \(2^B\). Together with the two atoms $U$, this part has total weight
		\[
		(q^2z^{-1})z^2=q^2z,
		\]
		which can be seen as the weight of a marked $ A$-colored part of size 2. 
		
		\item[(4)] \emph{Apply \(\mathsf B_r\) to 	\((\alpha,\widehat{r},\lambda)=(\varnothing,\widehat2,(5,3))\).}
		Since \(r-1=1\) occurs in \(\operatorname{Top}(\lambda)=\{4,1\}\), we applies the
		second branch of \(\mathsf B_2\) which raises the top-row entry \(1\)
		to \(2\), and leaves the bottom row unchanged,  which leads to
		\[
		\begin{pmatrix}4&1\\1&0\end{pmatrix}
		\longmapsto
		\begin{pmatrix}4&2\\1&0\end{pmatrix}.
		\]
		The latter Frobenius symbol corresponds to \(\mu=(5,4)\).  Consequently,
		\[
		\mathsf B_2(\varnothing,\widehat2,(5,3))=((1),(5,4)),
		\]
		so \(\alpha'=(1)\) and \(\mu=(5,4)\).  Notice that
		\(\operatorname{Top}(\mu)=\{4,2\}\) omits \(1\), and hence
		\(\mu\in\mathcal F_1\).
		
		\item[(5)] \emph{Decode at the new level $\mathsf K_{r-1}^{-1}(\mu)$.}
		At level \(1\), 
		\(\operatorname{mex}_1(5,4)=2\), so \((5,4)\in\mathcal M_1\).  Therefore
		\[
		\gamma'=\mathsf K_1^{-1}(\mu)
		=\mathsf K_1^{-1}(5,4)=(5,4).
		\]
		
		\item[(6)] \emph{Attach $\delta_{r-1}$ to derive $\beta'=\delta_{r-1}\sqcup\gamma'$.}
		Adjoining \(\delta_1=(1)\) gives
		\[
		\beta'=\delta_1\sqcup\gamma'=(5,4,1).
		\]
		Since \(\operatorname{mex}(\beta')=2=1+2\cdot0+1\), the output $(\alpha',\beta')$ belongs to
		\(\mathcal C_1\).  Altogether, we have 
		\[
		\Theta_2(U,U;\varnothing,(4,4,2,1))=((1),(5,4,1)).
		\]
	\end{itemize}
	
	Note that the above process can be reversed step by step to  recover the original elements.
	Under the injection $\Theta_r$, the bidegrees contributes to the generating functions can be checked directly. As the above example, the initial element in the invariant subset $\mathcal{C}_2$ has weight
	\[
	q^{|\beta|}z^{-\ell(\beta)}=q^{11}z^{-4},
	\]
	and  multiplied by the weight of the two atoms gives \(q^{11}z^{-2}\).
	The target element in $\mathcal{C}_1$ has weight
	\[
	q^{|\alpha'|+|\beta'|}z^{\ell(\alpha')-\ell(\beta')}
	=q^{1+10}z^{1-3}=q^{11}z^{-2}.
	\]
	Thus the total weight is preserved. 
	
	Now, we are ready to prove our main result. 
	
	\begin{proof}[Proof of Theorem~\ref{thm:main}]
		It is directly to see that 
		\[
		1+z+\cdots+z^{2k}
		=\sum_{m=0}^{2k-2}z^m+z^{2k-1}(1+z).
		\]
		For $a\ge0$, denote $r_a=k+2a+1$.  In \eqref{eq:com-decomposition}, the
		negative blocks occur exactly at  $r=r_a$ for some $a$ and we pair them with the preceding positive blocks at $r_a-1$ for $a>0$. When $a=0$, the needed positive block comes from $z^{2k-1}(1+z)$ multiplied by $C_k(z,q)$.  Hence
		\[
		\begin{aligned}
			J_k(z,q)
			={}&\left(\sum_{m=0}^{2k-2}z^m\right)C_k(z,q)+\sum_{a=0}^{\infty}z^{2r_a-3}(1+z)
			\Bigl(C_{r_a-1}(z,q)-z^2C_{r_a}(z,q)\Bigr).
		\end{aligned}
		\]
		For $k\geq 1$, each $C_k(z,q)$ is a  generating
		function with nonnegative coefficients, and each difference in the parentheses is nonnegative by using
		\eqref{eq:core-coeff-ineq}.  Thus $J_k(z,q)$ is coefficientwise
		nonnegative.
	\end{proof}

	By  Theorem \ref{thm:main} and Corollary \ref{cor:merca}, we conclude this section with some partition inequalities.  Let $p_2(n,s)$ be
	the number of ordered pairs of partitions $(\alpha,\beta)$ satisfying
	$|\alpha|+|\beta|=n$ and
	$\ell(\alpha)-\ell(\beta)=s$. 
	For $1\le S<R$, we let $p_{R,S}(N)$ count two-colored partitions of $N$ with
	$A$-colored parts being congruent to $S\pmod R$ and  $B$-colored parts being congruent to $-S\pmod R$.  Set  $p_2(n,s)=0$ if $n<0$ and	$p_{R,S}(N)=0$ for $N<0$. By extracting the coefficient $[q^nz^s]$ of $J_k(z,q)$ in
	Theorem~\ref{thm:main} and the coefficient $[q^N]$ of $\mathcal T_{R,S,k}(q)$ in Corollary \ref{cor:merca}, respectively, we obtain the following inequalities. 
	
	\begin{corollary}[Truncated partition inequalities]\label{cor:partition-inequalities}
		For $k\ge1$, $n\ge0$, and $s\in\mathbb Z$,
		\[
		\sum_{j=k}^{\infty}(-1)^{j-k}
		\sum_{u=-j}^{j}p_2(n-T_j,s-u)\ge0.
		\]
		For $1\le S<R$ and $N\ge0$,
		\[
		\sum_{j=k}^{\infty}(-1)^{j-k}
		\Bigl(
		p_{R,S}(N-RT_j+Sj)
		-p_{R,S}(N-RT_j-(j+1)S)
		\Bigr)\ge0.
		\]		
	\end{corollary}

	\section{Concluding remarks}\label{chap:conclu}
	
	Our proof divides the coefficientwise positivity problem into two combinatorial
	stages.  First, the involution on quadruples transforms our target into invariant subsets according to the mex.	Second, the injection \(\Theta_r\) maps each negative invariant subset associated to $z^2C_r(z,q)$  into	the preceding positive one associated to  $C_{r-1}(z,q)$ which preserves the bidegree of \((q,z)\) in the generating functions.  
	A natural next question is whether the invariant subset inequality
	\[
	z^2C_r(z,q)\preceq C_{r-1}(z,q)
	\]
	admits a direct map on the colored pairs \((\alpha,\beta) \) which gives an explicit partition-theoretic interpretation of the coefficientwise  positivity of 
	\(C_{r-1}(z,q)-z^2C_r(z,q)\). Moreover, it is also interesting to consider which kind of truncated theta series admits an analogous generalized-mex--Frobenius injection.
	
	\appendix
	\section{Konan's Bijection}\label{chap:appendix}
	
	This appendix is included  to  introduce the detailed procedures of Konan's bijection  $\mathsf K_r:\mathcal M_r\longrightarrow\mathcal F_r$ on the following two sets
	\[
	\mathcal M_r=\{\lambda:\operatorname{mex}_r(\lambda)-r\text{ is odd}\},\qquad
	\mathcal F_r=\{\lambda:r\notin\operatorname{Top}(\lambda)\},
	\]  which can be found in
	\cite[Sections~2--4]{Konan2023}.
	
	\subsection{The forward and inverse maps}
	
	We adopt the notation of the staircase as given in  \cite[Section 2]{Konan2023}:
	\[
	\Delta_{r,t}=(r+t,r+t-1,\ldots,r+1),
	\qquad \Delta_{r,0}=\varnothing.
	\]
	It is also compatible with the terminal state used in the bijection.
	
	For a partition $\lambda$, denote $\lambda_0=\infty$ and
	\[
	d_s(\lambda)=\max\{i\in\{0,1,\ldots,\ell(\lambda)\}:\lambda_i-i\ge s\}.
	\]
	Then $s\in\operatorname{Top}(\lambda)$ exactly when $d_s(\lambda)\ge1$ and
	$\lambda_{d_s(\lambda)}-d_s(\lambda)=s$.
	
	Let $\gamma\in\mathcal M_r$ and write
	\[
	\operatorname{mex}_{r}(\gamma)=r+2t+1,
	\qquad
	\gamma=\Delta_{r,2t}\sqcup\lambda,
	\]
	where the part $r+2t+1$ is absent from $\lambda$.  For
	$(\Delta_{r,2t},\lambda)$, set $s=r+2t$ and $d=d_s(\lambda)$.  Konan's
	forward map is iterated according to the following two moves.
	
	\medskip
	\noindent\textbf{Move K1: $s\in\operatorname{Top}(\lambda)$.}
	Let $\lambda_d=d+s$.  Delete this critical part, add one to the first
	$d-1$ parts, and insert a part $s+1$:
	\[
	\begin{aligned}
		&(\Delta_{r,2t};
		\lambda_1,\ldots,\lambda_{d-1},d+s,
		\lambda_{d+1},\ldots)\\
		&\qquad\longmapsto
		(\Delta_{r,2t};
		\lambda_1+1,\ldots,\lambda_{d-1}+1,s+1,
		\lambda_{d+1},\ldots).
	\end{aligned}
	\tag{K1}
	\]
	Under this move,  the staircase keeps unchanged, and the residual partition corresponding to $\lambda$ keeps both the size
	and the length.
	
	\medskip
	\noindent\textbf{Move K2: $s\notin\operatorname{Top}(\lambda)$ and $t\ge1$.}
	Remove the two largest parts $s,s-1$ from the staircase, subtract one from the first
	$d$ residual parts, and insert $d+s$ and $s-1$:
	\[
	\begin{aligned}
		&(\Delta_{r,2t};\lambda_1,\ldots,\lambda_d,
		\lambda_{d+1},\ldots)\\
		&\qquad\longmapsto
		(\Delta_{r,2t-2};
		\lambda_1-1,\ldots,\lambda_d-1,d+s,s-1,
		\lambda_{d+1},\ldots),
	\end{aligned}
	\tag{K2}
	\]
	with the residual parts reordered.  Under this move, the size of $\lambda$  increases by $s+(s-1)$ and its
	length increases by two, exactly offsetting the removed parts of the staircase.
	
	Iteration stops at a state $(\Delta_{r,0},\nu)$ with
	$r\notin\operatorname{Top}(\nu)$.  Konan sets $\mathsf K_r(\gamma)=\nu$ and proves that
	this procedure terminates and has an inverse.  
	
	\subsection{Examples}
	
	We first show the example corresponding to the one after Theorem \ref{thm:core-injection}. Let $r=2$ and $\gamma=(4,4)$.  Since
	$\operatorname{mex}_{2}((4,4))=3$,  we have $t=0$ and the shifted staircase $\Delta_{r,2t}$ is empty.  The top row of
	$(4,4)$ is $(3,2)$, so K1 should be applied with $s=2$ and $d=2$. We delete the second
	part $4$, increase the first part to $5$  and insert $3$.  Thus
	\[
	\mathsf K_2(4,4)=(5,3).
	\]
	Both partitions have size $8$ and length $2$, and
	$\operatorname{Top}(5,3)=\{4,1\}$ avoiding $2$. See Figure \ref{fig:Konan-one-step}.
	
	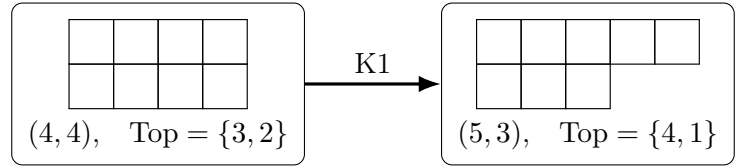
\begin{figure}[h]
		\centering
		\begin{tikzpicture}[>=Latex,node distance=18mm]
			\node[draw,rounded corners,align=center,inner sep=6pt] (a)
			{\ydiagram{4,4}\\[3pt] $(4,4)$,\quad $\operatorname{Top}=\{3,2\}$};
			\node[draw,rounded corners,align=center,inner sep=6pt,right=of a] (b)
			{\ydiagram{5,3}\\[3pt] $(5,3)$,\quad $\operatorname{Top}=\{4,1\}$};
			\draw[->,very thick] (a)--node[above]{K1}(b);
		\end{tikzpicture}
		\caption{A one-step instance of Konan's K1 move.  The operation preserves
			size and length while removing the forbidden Frobenius arm $2$.}
		\label{fig:Konan-one-step}
	\end{figure}
	
	We further consider a more involved example. Take
	\[
	r=2,
	\qquad
	\gamma=(6,5,4,3).
	\]
	Since
	\[
	\operatorname{mex}_{2}({\gamma})=7=2+2\cdot2+1,
	\]
	we have \(t=2\), and the initial decomposition is
	\[
	\gamma=\Delta_{2,4}\sqcup \varnothing,
	\qquad
	\Delta_{2,4}=(6,5,4,3).
	\]
	
	At the initial state, \(s=r+2t=6\) and
	\(d_6(\varnothing)=0\).  Since \(6\notin \operatorname{Top}(\varnothing)\), K2
	removes the two largest parts \(6\) and \(5\) of the staircase, and inserts
	these two parts into the residual partition.  Hence
	\[
	(\Delta_{2,4},\varnothing)
	\xrightarrow[\;s=6,\ d=0\;]{\mathrm{K2}}
	(\Delta_{2,2},(6,5)).
	\]
	
	At the second state,
	\[
	s=4,
	\qquad
	d_4(6,5)=1,
	\qquad
	\operatorname{Top}(6,5)=\{5,3\}.
	\]
	Thus \(4\notin \operatorname{Top}(6,5)\), and we apply K2 once more. This deletes the
	staircase parts \(4\) and \(3\), decreases the first residual part
	\(6\) to \(5\), and inserts
	\[
	d+s=5,
	\qquad
	s-1=3.
	\]
	After reordering the residual parts, this gives
	\[
	(\Delta_{2,2},(6,5))
	\xrightarrow[\;s=4,\ d=1\;]{\mathrm{K2}}
	(\varnothing,(5,5,5,3)).
	\]
	
	The shifted staircase is now empty, but the algorithm has not yet
	terminated.  Note that 
	\[
	\operatorname{Top}(5,5,5,3)=\{4,3,2\},
	\]
	so the forbidden Frobenius arm \(r=2\) is still present.  Here
	\[
	s=2,
	\qquad
	d_2(5,5,5,3)=3,
	\qquad
	\lambda_3=5=d+s.
	\]
	Using move K1,  delete the third part \(5\), increases the first
	two parts from \(5\) to \(6\), and inserts \(s+1=3\).  Consequently, we have
	\[
	(\varnothing,(5,5,5,3))
	\xrightarrow[\;s=2,\ d=3\;]{\mathrm{K1}}
	(\varnothing,(6,6,3,3)).
	\]
	
	Finally,
	$ 
	\operatorname{Top}(6,6,3,3)=(5,4,0)
	$,
	which does not contain \(2\).  Thus the resulting partition  
	belongs to \(\mathcal F_2\), and
	\[
	\mathsf K_2(6,5,4,3)=(6,6,3,3).
	\]
	The preservation of  size and   length can also   be observed  directly from the
	complete execution of Konan's forward map for 			\(r=2\) and \(\gamma=(6,5,4,3)\) as illustrated in Figure \ref{fig:Konan-complete-orbit}. 
	\begin{figure}[h!]
		\centering
		\resizebox{.74\textwidth}{!}{%
			\begin{tikzpicture}[>=Latex,node distance=13mm and 20mm]
				
				\node[draw,rounded corners,align=center,inner sep=7pt] (s0)
				{\ydiagram{6,5,4,3}\quad;\quad $\varnothing$\\[4pt]
					$(\Delta_{2,4},\varnothing)$};
				
				\node[draw,rounded corners,align=center,inner sep=7pt,right=of s0] (s1)
				{\ydiagram{4,3}\quad;\quad\ydiagram{6,5}\\[4pt]
					$(\Delta_{2,2},(6,5))$};
				
				\node[draw,rounded corners,align=center,inner sep=7pt,below=of s1] (s2)
				{$\varnothing$\quad;\quad\ydiagram{5,5,5,3}\\[4pt]
					$(\varnothing,(5,5,5,3))$\\[-1pt]
					$\operatorname{ Top}=\{4,3,2\}$};
				
				\node[draw,rounded corners,align=center,inner sep=7pt,left=of s2] (s3)
				{$\varnothing$\quad;\quad\ydiagram{6,6,3,3}\\[4pt]
					$(\varnothing,(6,6,3,3))$\\[-1pt]
					$\operatorname {Top}=\{5,4,0\}$};
				
				\draw[->,very thick]
				(s0)--node[above,align=center]
				{\(\mathrm{K2}\)\\[-1pt]\(\scriptstyle s=6,\ d=0\)}(s1);
				
				\draw[->,very thick]
				(s1)--node[right,align=center]
				{\(\mathrm{K2}\)\\[-1pt]\(\scriptstyle s=4,\ d=1\)}(s2);
				
				\draw[->,very thick]
				(s2)--node[above,align=center]
				{\(\mathrm{K1}\)\\[-1pt]\(\scriptstyle s=2,\ d=3\)}(s3);
				
		\end{tikzpicture}}
		\caption{The complete orbit of Konan's forward map for
			\(r=2\) and \(\gamma=(6,5,4,3)\). }
		\label{fig:Konan-complete-orbit}
	\end{figure}
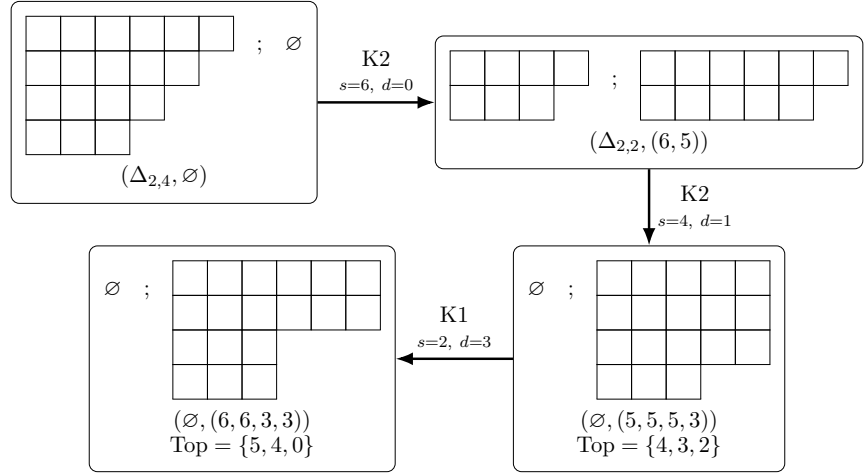
	
	\section*{Acknowledgments}
	This work is supported by the National Natural Science Foundation of China (Grant No. 12571351, 12071235), Tianjin Natural Science Foundation (No. 24JCZDJC01390) and the Fundamental Research Funds for the Central Universities of China.
	
	\section*{Declaration of competing interest}
	The authors declare that they have no known competing financial interests or personal relationships that could have appeared to influence the work reported in this paper.
	

\end{document}